\UseRawInputEncoding
\documentclass{amsart}
\usepackage{amssymb,amscd}
\usepackage{graphicx}
\usepackage[T1]{fontenc}
\usepackage[english]{babel}
\usepackage{color} \usepackage[usenames,dvipsnames]{xcolor}
\usepackage{amsthm,amssymb,amsmath}
\usepackage{enumerate}
\usepackage{mathtools}
\usepackage{bm,array}

 \newtheorem{thm}{Theorem}[section]
  \newtheorem{thmA}{Theorem}

 \newtheorem{cor}[thm]{Corollary}
 \newtheorem{lemma}[thm]{Lemma}
 \newtheorem{prop}[thm]{Proposition}
 \theoremstyle{definition}
 \newtheorem{defn}[thm]{Definition}

 \theoremstyle{remark}
 \newtheorem{rem}[thm]{Remark}

 \numberwithin{equation}{subsection}
 \newtheorem{ack}{Acknowledgment}

\newcommand{\cT}{\text{$\mathcal{T}$}}

\newcommand{\cE}{\text{$\mathcal{E}$}}

\newcommand{\FF}{\text{$\mathcal{F}$}}

\newcommand{\cL}{\text{$\mathcal{L}$}}

\newcommand{\id}{\operatorname{id}}

\newcommand{\Homeo}{\operatorname{Homeo}}

\newcommand{\St}{\operatorname{St}}

\newcommand{\lvl}{\operatorname{Lvl}}

        \newcommand{\field}[1]{\text{$\mathbb{#1}$}}
        \newcommand{\N}{\field{N}}

        \newcommand{\R}{\field{R}}

\newdimen\theight
\def\TeXref#1{%
             \leavevmode\vadjust{\setbox0=\hbox{{\cT
                     \quad\quad  {\small \textrm #1}}}%
             \theight=\ht0
             \advance\theight by \lineskip
             \kern -\theight \vbox to
             \theight{\rightline{\rlap{\box0}}%
             \vss}%
             }}%

\makeatletter
\def\blfootnote{\gdef\@thefnmark{}\@footnotetext}
\makeatother

\begin{document}

 \title[Minimal Hyperbolic Foliations with non simply-connected generic leaf]{Leaf topology of minimal hyperbolic foliations with non simply-connected generic leaf}                                                                                                                                                                                                               

\author{Paulo Gusm\~ao$^\dagger$}
\author{Carlos Meni\~no Cot\'on$^\ddagger$}

\blfootnote{\textup{2010} \textit{Mathematics Subject Classification}: Primary 57R30}
\maketitle

\address {$\dagger$ Departamento de An\'alise, Instituto de Matem\'atica e Estat\'istica, Universidade Federal Fluminense.\ 
	
\address{$\ddagger$ CITMAGA \&
Departamento de Matem\'atica Aplicada I, Universidade de Vigo\\ 
Instituto Innovacións Tecnolóxicas, CP 15782 Santiago de Compostela \& Escola de Enxe\~ner\'ia Industrial, Rua Conde de Torrecedeira 86, CP 36208, Vigo, Spain.}\\

\email{email: phcgusmao@id.uff.br \& carlos.menino@uvigo.es}

\begin{abstract}
A noncompact (oriented) surface satisfies the condition $(\star)$ if
{\it their isolated ends are accumulated by genus}. We show that every  surface satisfying this condition is homeomorphic to the  leaf   of a  minimal codimension one foliation on a closed $3$-manifold whose generic leaf is not simply connected. Moreover, the above result is also true for any prescription of a countable family of noncompact surfaces (satisfying $(\star)$): they can coexist in the same minimal codimension one  foliation as above. All the given examples are hyperbolic foliations, meaning that they admit a leafwise Riemannian metric of constant negative curvature.
\end{abstract}

\section{Introduction}
It is well known that every surface is homeomorphic to a leaf of a $C^\infty$ codimension one foliation on a compact $3$-manifold (see \cite{Cantwell-Conlon-1987}). In \cite{Cantwell-Conlon-1987}, the exotic leaf topology appears at the infinite level of the foliations and it is obtained as a Hausdorff limit of finite level leaves (see \cite[Chapter 5]{Candel-Conlon-I-2000} for an overview on the theory of levels). Every foliation with more than one level cannot be minimal, this opens the question about what noncompact surfaces can be homeomorphic to leaves of a minimal foliation. Observe that any codimension one oriented foliation on a $3$-manifold with no transverse invariant measure (and that is the usual case) admits a leafwise hyperbolic metric varying continuously in the ambient space \cite{Candel1}. Thus, it is natural to study this problem on minimal hyperbolic foliations.

The study of minimal hyperbolic foliations on closed $3$-manifolds is extensively treated in \cite{ADMV2} but the possible topologies of leaves that can occur on the given foliations is not (completely) treated in that work. 

In the recent work \cite{ABMP}, S. \'Alvarez, J. Brum, M. Mart\'inez, R. Potrie provide an interesting construction of a minimal hyperbolic lamination on a compact space where all the noncompact oriented surfaces are realized (topologically) as leaves.

Using completely different techniques \cite{Gusmao-Menino} the authors prove that for every countable family of noncompact oriented surfaces, there exists a transversely minimal hyperbolic foliation $\FF$ on some Seifert closed $3$-manifolds, the foliation being transverse to its fibers, and such that any surface of the family appears as a leaf. In both works the generic leaves are planes. Here generic means from the topological point of view, i.e., a residual set formed by homeomorphic leaves \cite{Cantwell-Conlon-1998}. The generic leaf of any minimal hyperbolic foliation on any closed $3$-manifolds can be one of the following:
\begin{itemize}
\item[(a)] a plane,

\item[(b)] a Loch-Ness monster (a plane with infinitely many handles attached),

\item[(c)] a cilinder,

\item[(d)] a Jacob's ladder\footnote{In this work, Jacob's ladder is always considered with $2$ ends, the one-ended Jacob's ladder (pictured in Figure~\ref{f:figure_3}) is homeomorphic (but not quasi-isometric) to the Loch Ness monster; since we are only interested in topology we shall use ``Loch Ness monster'' for any one-ended surface with nonplanar end.} (a cilinder with infinitely many  handles attached accumulating to both ends),

\item[(e)] a Cantor tree (the sphere minus a Cantor set),

\item[(f)] a Cantor tree with handles (a Cantor tree with a handle attached to each bifurcation, thus every end is accumulated by genus)\;.
\end{itemize}

We recall the following result that is a direct consequence from Theorems $4$ and $2$ in \cite{ADMV}.

\begin{prop}\cite{ADMV}\label{p:systolic loop}
If a leaf of a minimal hyperbolic foliation on a closed $3$-manifold contains a nontrivial loop which supports trivial holonomy, then each end of any leaf is not simultaenously planar and isolated. 
\end{prop}

This suggests the following definition.

\begin{defn}[Condition $(\star)$]
A noncompact oriented surface satisfies the {\em Condition} $(\star)$ if its isolated ends are nonplanar.
\end{defn}

Recall that the leaves without holonomy form a $G_\delta$ set (see \cite{E-M-T} and  \cite{Hector}). Therefore, from Proposition~\ref{p:systolic loop} we get the following Corollary.

\begin{cor} 
If the generic leaf of a minimal hyperbolic lamination is not simply connected (i.e., a plane), then all of its leaves satisfy the Condition $(\star)$. Moreover, if the generic leaf has positive genus, then every end of every leaf is accumulated by genus.
\end{cor}

In a more recent work \cite{AB}, S. \'Alvarez \& J. Brum describe the allowed topologies for non-generic leaves of laminations that can  occur when the topological type of the generic leaf is given, showing that Condition $(\star)$ is a necessary and sufficient condition.

\begin{thm} \cite{AB} \label{t:Alvarez-Brum}
Every surface that satisfies the Condition $(\star)$ is homeomorphic to a leaf of a minimal hyperbolic lamination whose generic leaf is a Cantor tree. Moreover, for any arbitrary (possibly uncountable) prescription of noncompact oriented surfaces satisfying the Condition $(\star)$, there exists a minimal hyperbolic lamination whose generic leaf is the Cantor tree and such that it contains the prescribed surfaces as leaves. The given hyperbolic lamination can be obtained as a suspension of homeomorphisms acting on a Cantor set.
\end{thm}

When the generic leaf has handles, the condition on leaves is strenghtened. 

\begin{defn}[Condition $(\star \star)$]
A noncompact oriented surface satisfies the {\em Condition} $(\star \star)$ if every end is accumulated by genus.
\end{defn}


\begin{cor} \cite{AB}  
Every surface that satisfies the Condition $(\star \star)$  is homeomorphic to a leaf of a minimal hyperbolic lamination whose generic leaf is a Cantor tree with handles. Moreover, for any arbitrary (possible uncountable) prescription of noncompact oriented surfaces satisfying the Condition $(\star\star)$, there exists a minimal hyperbolic lamination whose generic leaf is the Cantor tree  with handles and such that it contains the prescribed surfaces as leaves. The given hyperbolic lamination can be obtained as a suspension of homeomorphisms acting on a Cantor set.
\end{cor}

This Corollary follows from Theorem~\ref{t:Alvarez-Brum} by means of a surgery along a transverse fiber, changing a disk by a handle at each point of that fiber. 


In  \cite[Theor\`eme 3]{Blanc} Blanc gives a complete description of the open surfaces that can be realized as leaves of  minimal laminations where the generic leaf has two ends: all leaves have one or two ends. If in addition the foliation is hyperbolic then  Proposition~\ref{p:systolic loop} implies that the generic leaf is the Jacob's ladder and the unique topologies that can occur as leaves are the Jacob's ladder and the Loch-Ness monster.  

The Table~\ref{tb:table_1} (depicted also in \cite{AB}) gives the necessary conditions for the topology of every leaf of a minimal hyperbolic lamination  depending on the topology of the generic leaf. Recall that hyperbolicity implies that all the leaves must be oriented.

\begin{table}[h!]
\begin{center}
\begin{tabular}{ | m{4cm} | m{6cm}| } %
  \hline
  \textbf{Generic Leaf} & \textbf{Leaf Topology}  \\ 
  \hline
  Plane & No Conditions  \\ 
  \hline
  Cantor Tree & Condition $(\star)$  \\ 
  \hline
 Loch Ness Monster or Cantor Tree  with handles & Condition $(\star\star)$ \\
  \hline
  Jacob's Ladder & Jacob's Ladder or Loch Ness Monster\\
  \hline
\end{tabular}
\caption{}
\label{tb:table_1}
\end{center}
\end{table}

We studied the first case of Table~\ref{tb:table_1} in  \cite{Gusmao-Menino} showing that every countable family of noncompact oriented surfaces can be realized as leaves of a codimension one minimal hyperbolic foliation on some closed $3$-manifolds.

In this work we deal with the intermediate cases of Table~\ref{tb:table_1}, these can also be realized as leaves of some codimension one minimal hyperbolic foliations on some closed $3$-manifolds, this is summarized in the following theorems.

\begin{thmA}  \label{t:thm1}
Let $\{S_n\}_{n\in\mathbb N}$ be a countable family of noncompact oriented surfaces satisfying the Condition $(\star)$. There exists a   transversely $C^\infty$ minimal hyperbolic foliation $\FF$ of a closed $3$-manifold whose generic leaf is a Cantor tree and such that for each $n\in \mathbb N$ there exists a leaf in $\FF$ homeomorphic to $S_n$.
\end{thmA}
  
\begin{thmA}\label{t:thm2}
Let $\{S_n\}_{n\in\mathbb N}$ be a countable family of noncompact oriented surfaces satisfying the Condition $(\star \star)$. There exist hyperbolic foliations $\FF_1$, $\FF_2$ on (possibly different) closed $3$-manifolds whose generic leaf is the Cantor tree with handles and the Loch Ness monster (respectively) and such that for each $n\in \mathbb N$ there exists a leaf in $\FF_i$ homeomorphic to $S_n$, for $i\in\{1,2\}$. The transverse regularity of $\FF_1$ is  $C^\infty$ and $\FF_2$ is just Lipschitz.
\end{thmA}

The foliations given in Theorem~\ref{t:thm1} are obtained via surgery on minimal foliated manifolds. These are obtained just by removing tubular neighborhoods from transverse circles of two minimal foliations whose generic leaves are planes and then gluing boundaries with an appropriate homeomorphism. The point here is the definition of the gluing map in order to control the topology of countably many prescribed leaves as well as the generic leaf. The topology of the ambient $3$-manifold can be controlled and includes several of the cases studied in \cite{ADMV}, more precisely, some Seifert and graph manifolds admit foliations as those described in Theorems~\ref{t:thm1} and \ref{t:thm2}. Since our construction produces an incompressible torus, the ambient manifold is never hyperbolic. Theorem~\ref{t:thm2} is a corollary of Theorem~\ref{t:thm1} and our work \cite{Gusmao-Menino}.

There are lots of examples of minimal foliations on closed $3$-manifolds whose generic leaf is a plane, for instance: minimal Kr\"onecker foliations of $T^3$ obtained by suspension of rationally independent rotations, the center-stable foliation of a transitive Anosov flow, the foliation given by a $C^2$ locally free  action of affine group on a $3$-manifold \cite{Ghys3}, the supension of a minimal action of a surface group over $\Homeo_+(S^1)$  with maximal Euler class \cite{Ghys4}. The generic leaves of the examples contructed in \cite{Gusmao-Menino} are also planes.

The last case of Table~\ref{tb:table_1} is handled in \cite[Section 2]{Blanc}, i.e., there exists a minimal hyperbolic foliation on some closed $3$-manifold whose generic leaf is the Jacob's ladder and there exist leaves homeomorphic to the Loch Ness monster. Blanc's example produces exactly $4$ nongeneric leaves, it is unknown for us if it is possible to produce a similar example with infinitely many nongeneric leaves.

The paper is organized as follows.

\begin{itemize}

\item In the first section we describe how to realize any noncompact oriented surface satisfying the condition $(\star)$ as an inductive limit of suitable boundary gluings between surfaces homeomorphic to a plane punctured along countably many disks.

\item In the second section we introduce the so called {\em Foliated blocks} and study the gluing maps between them. The leaves of the resulting foliation can be seen as inductive limits as the described in the first section and the aim is to define a suitable gluing map  that controls the topology of countably many leaves.

\item In the third section we prove Theorems~\ref{t:thm1} and \ref{t:thm2} and discuss the topology of the ambient $3$-manifolds where these Theorems hold.
\end{itemize}


\section{Realization of surfaces satisfying the condition $(\star)$}

Let $S$ be an noncompact connected surface (without boundary) and let $\{K_i\}_{i\in\N}$ be a compact filtration of $S$, i.e., $K_i\subset K_{i+1}$ for all $i\in\N$ and $\bigcup_i K_i = S$. 

A system of neighborhoods of an {\em end} of $S$ is a family $U_1\supset U_2\supset\cdots \supset U_n\cdots$ of connected open sets such that each $U_i$ is a connected component of the complement $K_i$. An {\em end} is an equivalence class of neighborhood systems (relative to possibly different filtrations): two systems are equivalent if and only if the interesection between every pair of neighborhoods in each system is nonempty.

The space of ends will be denoted by ${\cE}(S)$ and admits a topology of a compact, metrizable and totally disconnected space.

An end $e$ is said to be nonplanar if none of its fundamental neighborhoods
$U_i$ is homeomorphic to an open subset of $\R^2$. The set of nonplanar ends is a closed subspace $\cE^*(S)$ of $\cE(S)$. When $S$ is orientable, the classification theorems of B. Kerékjártó \cite{kerekjarto} and I. Richards \cite{Richards} can be stated as follows.

\begin{prop}\label{p:classification} \cite{kerekjarto, Richards}
Let $S$ be an oriented noncompact oriented surface. If $\cE^*(S)\ne\emptyset$, then the homeomorphism type of the pair $(\cE(S)$, $\cE^*(S))$ determines $S$ up to diffeomorphism.
If $\cE^*(S)=\emptyset$, then the genus of $S$ together with the homeomorphism type of $\cE(S)$ determine $S$ up to diffeomorphism. Finally, every topological pair $(E, E^*)$ such that $E$ is a compact, totally disconnected, metrizable space and $E^*$, a closed subspace, occurs as $(\cE(S), \cE^*(S))$ for some noncompact oriented surface $S$. If $\cE^*(S)=\emptyset$, any preassigned integer can be realized as genus($S$).
\end{prop}

  Let $\Sigma$ be the surface obtained by removing from $\R^2$ the interior of a proper union of countably many pairwise disjoint closed disks, here proper means that any bounded set meets finitely many disks of the family. The topology of $\Sigma$ does not depend on the chosen family of disks. The boundary components of $\Sigma$ will be enumerated as $B_i$, $i\in\N$ (see Figure~\ref{f:figure_1}).
  
\begin{figure}
\begin{center}
\includegraphics[scale=0.50]{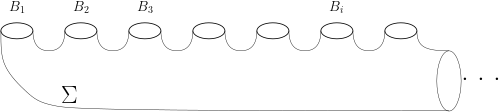}
\end{center}
\caption{}
\label{f:figure_1}
\end{figure}

In order to prove Theorem~\ref{t:thm1}, we need to obtain the topology of any noncompact oriented surface (without boundary) satisfying the Condition $(\star)$ from boundary unions of surfaces homeomorphic to $\Sigma$. 

\begin{defn}\label{d:left-right}
Let $\cL^\triangleright=\{L^\triangleright_n\mid\ n\in I^\triangleright\}$ and $\cL^\triangleleft=\{L^\triangleleft_n\mid\ n\in I^\triangleleft\}$ be two disjoint countable (possibly finite\footnote{This means that $I^\ast$ can refer to $\N$ or the set $\{1,\dots,n\}$ for some $n\in\N$.}) families of surfaces homeomorphic to $\Sigma$. Let $\{B^\ast_{n,k}\mid\ k\in\N\}$ be a enumeration of the boundary components of each $L^\ast_n$, $n\in\N$, for $\ast\in\{\triangleright,\triangleleft\}$.

Let $\delta:I^\triangleright\times\N\to I^\triangleleft\times\N$ be any bijection and let $g_{i,j}:B^\triangleright_{i,j}\to B^\triangleleft_{\delta(i,j)}$ be reversing orientation homeomorphisms. Let $S_\delta$ be the orientable open surface without boundary obtained as the quotient of $\bigsqcup_n L^\triangleright_n\sqcup \bigsqcup_m L^\triangleleft_m$ by the gluing maps $g_{i,j}$ between the boundary components of the surfaces in $\cL^\triangleright$ and $\cL^\triangleleft$.
\end{defn}

\begin{prop}\label{p:surface construction}
Let $S$ be a noncompact oriented surface satisfiyng the Condition $(\star)$, there exist two countable families $\cL^\triangleright=\{L^\triangleright_n\mid\ n\in I^\triangleright\}$ and $\cL^\triangleleft=\{L^\triangleleft_n\mid\ n\in I^\triangleleft\}$  of surfaces homeomorphic (preserving orientation) to $\Sigma$ and a bijection $\delta:I^\triangleright\times\N\to I^\triangleleft\times\N$ such that $S_\delta$ is homeomorphic to $S$.
\end{prop}

We shall codify the topology of $S$ in a suitable way (this is essentially the same codification as the used in \cite{Gusmao-Menino}).

\begin{defn}\label{n:binary tree}
It is well known that the binary tree has a Cantor set of ends and every compact, metrizable and totally disconnected space can be embedded in a Cantor set. Every closed subset of the Cantor set can be obtained as the space of ends of a connected subtree $\mathcal{T}$ of the binary tree. Let $V$ and $E$ be the sets of vertices and edges, respectively, of $\mathcal{T}$ and let $\nu: V\to \{0,1\}$ be any function, that will be called a {\em vertex coloring}. For each $v\in V$, let us define $\St(v)$ as the set of edges that contain $v$ in their boundaries. Let $\deg(v)=\#\St(v)$, that will be called the {\em degree}  of the vertex $v$. Since $\mathcal{T}$ is a subtree of the binary tree, it follows that $\deg(v)\leq 3$ for all $v\in V$. 

Without loss of generality we can assume that the root element of the binary tree belongs to $\mathcal{T}$. This root element will be denoted by $\bm{\mathring{v}}$. The vertices of a tree are partitioned by levels from a root element: level $0$ is just the root element, and, recursively, a vertex $v$ is of level $k$, denoted by $\lvl(v)=k$ if it is connected by an edge with some vertex of level $k-1$.

Define an orientation on the binary tree just by declaring the origin of an edge $e$ as the boundary vertex with lowest level, see Figure~\ref{f:figure_2}. Let $o(e)$, $t(e)$ denote the origin and target of an oriented edge.
\end{defn}

\begin{figure}
\begin{center}
\includegraphics[scale=0.4]{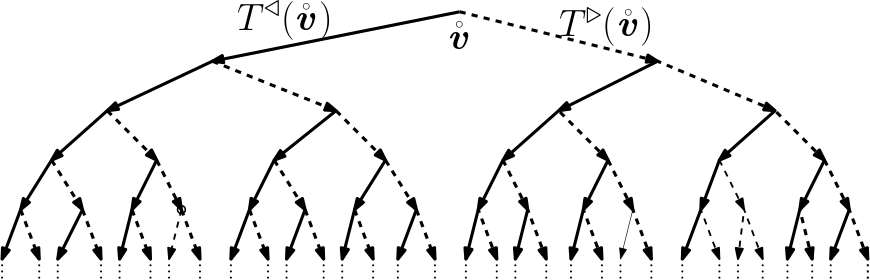}
\end{center}
\caption{The binary tree with the chosen orientation. Dashed (resp. bold) lines represent its right-sided (resp. left-sided) branches.}
\label{f:figure_2}
\end{figure}

\begin{defn}\label{d:step construction}
Let $\mathcal{T}$ be a subtree of the binary tree and let $\nu: V\to \{0,1\}$ be a vertex coloring. Let us define $K_v=S^2\setminus\bigsqcup_{e\in\St(v)}D_e^2$ if $\nu(v)=0$ and $K_v=T^2\setminus\bigsqcup_{e\in\St(v)}\mathring{D}_e^2$ if $\nu(v)=1$,  where $\{D_e^2\}_{e\in\St(v)}$ is a collection of $\deg{v}$ pairwise disjoint smooth disks in $S^2$ or $T^2$ respectively. 

Let $C^e_v$ denote the boundary component of $K_v$ obtained by removing the ball $\mathring{D}^2_e$, for $e\in\St(v)$. Define $S_{\mathcal{T},\nu}$ as the open orientable surface obtained from $\bigsqcup_{v\in V_F}K_v$ by attaching each boundary component $C^e_{o(e)}$ with  $C^e_{t(e)}$ via a reversing orientation homeomorphism.
\end{defn}

It is clear, from Proposition~\ref{p:classification}, that every noncompact oriented surface is homeomorphic to some $S_{\mathcal{T},\nu}$ for a suitable choice of $\mathcal{T}$ and $\nu$. We shall consider, without loss of generality, connected subtrees without ``dead ends'', i.e., vertices where $\deg{v}=1$ are forbidden. In this case the above process generates all the noncompact oriented surfaces with at least two ends. The unique noncompact oriented surface with one end satisfying the Condition $(\star)$ is the Loch Ness monster and will be treated separately.

\begin{defn}\label{d:decomposition}
A linear subtree (i.e., isomorphic to the Cayley graph of $\N$) of a tree will be called a {\em branch}.

We shall decompose any connected subtree of the binary tree (without dead ends and containing the root vertex) as a union of (oriented) branches.

The root vertex and any other vertex of degree $3$ are the origins of two edges. We declare one of this edges as {\em right-sided} and the other as {\em left-sided}\footnote{This prescription of ``left'' and ``right'' can be arbitrary but we shall use the planar embedding suggested by figure~\ref{f:Figure_2} for a visual picture of left-sided and right-sided edges as those that points to the left and right respectively.}. The target of a left-sided (resp. right-sided) edge is called a left-sided (right-sided) target from $v$. If $v$ has degree $2$ and is not the root vertex then there exists a unique edge with origin at $v$ and its target will be called the target from $v$, in this case the target is considered simultaneously as left and right-sided.

Let $\mathcal{T}$ be a connected subtree of the binary tree that contains $\bm{\mathring{v}}$ and no dead ends. Let $\mathcal{T}^\triangleleft(\bm{\mathring{v}})$ (resp.  $\mathcal{T}^\triangleright(\bm{\mathring{v}})$) be the branch of $\mathcal{T}$ whose initial vertex is $\bm{\mathring{v}}$ and defined inductively by the following property: it contains the left-sided (resp. right-sided) targets from any vertex in this graph. 

If $\deg(v)=3$ and $v \in \mathcal{T}^\triangleleft(w)$ (resp. $v \in \mathcal{T}^\triangleright(w)$) for some vertex $w$ then define  $\mathcal{T}^\triangleright(v)$ (resp. $\mathcal{T}^\triangleleft(v)$) as the branch with initial vertex at $v$ and defined inductively by the property: it contains the right-sided (resp. left-sided) target from any vertex in this graph.

It is clear that the edges of these branches form a partition of the set of edges of $\mathcal{T}$. Observe that for degree $3$ vertices only one branch, $\mathcal{T}^\triangleright(v)$ or $\mathcal{T}^\triangleleft(v)$, exists.
\end{defn}

\begin{defn}\label{d:construction with property star}
It is said that a coloring $\nu:V\to\{0,1\}$ satisfies the {\em Condition} ($\star$) if and only if every isolated end of $\mathcal{T}$ is accumulated by $1$'s. More precisely, if $V_\mathbf{e}$ is the set of vertices of a connected subgraph of $\cT$ which contains a neighborhood of an isolated end $\mathbf{e}$ then $1\in\nu(V_\mathbf{e})$.

It is clear that $S_{\mathcal{T},\nu}$ satisfies the Condition ($\star$) if and only if the vertex coloring $\nu$ satisfies the Condition ($\star$) just defined.
\end{defn} 

\begin{proof}[\bf{Proof of Proposition~\ref{p:surface construction}}]

Consider first the case with one end, this is the Loch Ness monster. Let $L^\triangleright, L^\triangleleft$ be surfaces homeomorphic to $\Sigma$. Let $\{B^\triangleright_n\mid\ n\in \mathbb \N\}$ and $\{B^\triangleleft_n\mid\ n\in \mathbb \N\}$ be respective enumerations of the boundary components of these surfaces. The open surface obtained by gluing $L^\triangleright$ and $L^\triangleleft$ by attaching the boundary $B^\triangleright_n$ with $B^\triangleleft_n$ has one end and infinite genus, therefore homeomorphic to the Loch-Ness monster (see figure~\ref{f:Figure_3}).

\begin{figure}[h]
\begin{center}
\includegraphics[scale=0.50]{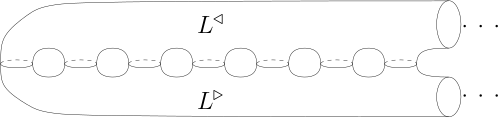}
\end{center}
\caption{Construction of the Loch Ness monster.}
\label{f:figure_3}
\end{figure}

Let us consider now the case of a surface $S$ with two or more ends satisfiyng the Condition ($\star$). Let $\mathcal{T}$ be a subtree of the binary tree and let $\nu:V\to\{0,1\}$ be a vertex coloring (satisfying also Condition $(\star)$) such that $S_{\mathcal{T},\nu}$ is homeomorphic to $S$.

For each branch $\mathcal{T}^\ast(v)$, $\ast\in\{\triangleright,\triangleleft\}$, defined in  Definition~\ref{d:decomposition}, let us consider a surface $L^\ast(v)$ homeomorphic (preserving orientation) to $\Sigma$. Similarly, for each isolated end  $\mathbf{e}$ (if there is any) of the tree $\cT$ let us consider a surface $L(\mathbf{e})$ homeomorphic to $\Sigma$.

Enumerate the boundary components of $L^\ast(v)$ and $L(\mathbf{e})$ as $\{B^\ast_n(v)\mid\ n\in \mathbb \N\}$, for some $\ast\in\{\triangleright,\triangleleft\}$ and $\{B_n(\mathbf{e})\mid\ n\in \mathbb \N\}$ respectively.

We attach now the boundary components of the previous surfaces, the resulting manifold will be homeomorphic to $S_{\mathcal{T},\nu}$. The construction will be inductive. At the step $k$ of the construction we shall obtain the construction of $S_{\mathcal{T},\nu}$ until the level $k$ of the tree $\mathcal{T}$.

For the step $0$ we need just the surfaces $L^\triangleright(\bm{\mathring{v}})$ and $L^\triangleleft(\bm{\mathring{v}})$. If $\nu(\mathring{\textbf{v}})=0$ then we attach $B^\triangleright_1(\mathring{\bm{v}})$ to $B^\triangleleft_1(\mathring{\bm{v}})$ via a reversing orientation homeomorphism. If $\nu(\mathring{\bm{v}})=1$ then we also attach the boundary components $B^\triangleright_2(\mathring{\bm{v}})$ with $B^\triangleleft_2(\mathring{\bm{v}})$. The given identification between these two surfaces reproduces the topology of $S_{\mathcal{T},\nu}$ at the unique vertex at level $0$ as desired.

Assume now that the construction was done until step $k$ and define the step $k+1$. Let $v_1,\dots, v_N$ be an enumeration of the vertices of $\cT$ at level $k+1$, we proceed now by finite induction on these vertices. 

Given $v_1$, there exists a unique vertex $w_1$ with level $\leq k$ such that $v_1\in \cT^{\bullet}(w_1)$ for some $\bullet\in\{\triangleright,\triangleleft\}$. Let $n_1$ be the first integer such that the boundary component $B^{\bullet}_{n_1}(w_1)$ was not already attached to another boundary component of other surface.

\begin{itemize}
\item If $\deg(v_1)=3$ and $\nu(v_1)=0$ then attach $B^{\bullet}_{n_1}(w_1)$ with $B^\ast_1(v_1)$, which is the first boundary component of $\cT^\ast(v_1)$, (observe that $\bullet$ and $\ast$ are different elements in $\{\triangleright,\triangleleft\}$). If $\nu(v_1)=1$ then we also attach $B^{\bullet}_{n_1+1}$ with $B^\ast_2(v_1)$ in order to produce a handle in this step.

\item If $\deg(v_1)=2$ and $\nu(v_1)=0$ then nothing needs to be done.

\item If $\deg(v_1)=2$, $\nu(v_1)=1$ and there exists a degree $3$ vertex in $\cT^{\bullet}(w_1)$ whose level is greater than $k+1$, then let $L^\ast(w)$ be the surface that has a component attached with $B^{\bullet}_{n_1-1}(w_1)$ and let $B^\ast_{n_w}(w)$ be the first boundary component in this surface that was still not attached to any other boundary component of other surface and attach $B^{\bullet}_{n_1}(w_1)$ with $B^\ast_{n_w}(w)$. This will produce a new handle in the construction.

\item If $\deg(v_1)=2$, $\nu(v_1)=1$ and there no exist vertices of degree $3$ whose level is greater than $k+1$, then this branch defines an isolated end $\mathbf{e}$. Let $\ell$ be the first integer such that the boundary component $B_{\ell}(\mathbf{e})$ of $L(\mathbf{e})$ was not already attached and then attach $B^{\bullet}_{n_1}(w_1)$ with $B_{\ell}(\mathbf{e})$. This will produce accumulation of genus to this end as required by Condition $(\star)$.
\end{itemize}

Continue the above process with $v_2,\dots, v_N$, and apply the same procedure with the next steps of the construction.

Observe that at each vertex with degree $3$ we include a new surface homeomorphic to $\Sigma$ that has only finitely many components attached with the surfaces used in the previous steps, therefore a bifurcation is effectively produced by our inductive construction.

Since $\nu$ satisfies property ($\star$) it follows that any linear subtree $\cT^\ast(v)$ contains infinitely many vertices with degree $3$ or infinitely many vertices with color $1$. This guarantees that the induction process never stops and eventually every boundary component of every surface $L^\ast(v)$ or $L(\mathbf{e})$ is attached to some other component of other surface. By Proposition~\ref{p:classification} the surface obtained by the previous boundary identifications is homeomorphic to $S_{\mathcal{T},\nu}$ as desired.

Relative to the precise statement of Proposition~\ref{p:surface construction}, the family $\{L^\triangleright_n\}$ will be an enumeration of the previous used surfaces whose superscript is $\triangleright$ and the surfaces $L(\mathbf{e})$ that are attached to a surface of the form $L^\triangleleft(v)$. The family $\{L^\triangleleft_n\}$ is defined analogously. The bijection $\delta$ is defined from the indices of the boundary components attached in the previous construction.
\end{proof}

\begin{figure}[h]
\begin{center}
\includegraphics[scale=0.29]{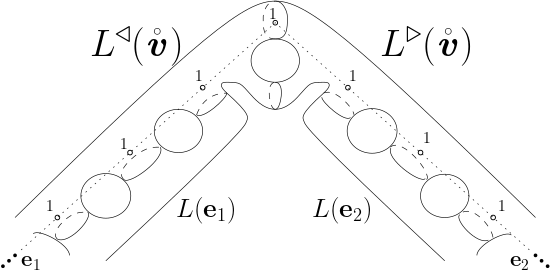}
\includegraphics[scale=0.25]{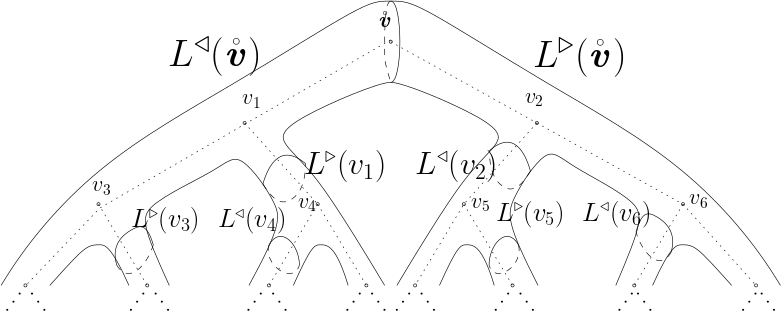}
\end{center}
\caption{Construction of the Jacob's ladder (left) and the Cantor Tree (right) from surfaces homeomorphic to $\Sigma$.}
\label{f:figure_4}
\end{figure}

\begin{rem}\label{r:independence}
The topology of $S_\delta$ does not depend in the enumeration of the boundary components of the surfaces $L^\triangleright_n$ and $L^\triangleleft_n$, $n\in\N$. This readily follows from the following fact: let $\tau:\N\to\N$ be a bijection, there exists a preserving orientation homeomorphism $f_\tau:\Sigma\to\Sigma$ such that $f_\tau(B_n)=B_{\tau(n)}$ for all $n\in\N$. The proof of this fact is easy and left to the reader.
\end{rem}

\section{Transverse gluings}

\begin{defn}[Foliated Block]
A {\em foliated block} is a minimal foliation on a compact $3$-manifold with a transverse boundary that consists in a torus. The generic leaf is the surface $\Sigma$ (the plane minus a countable family of balls) and the trace foliation of the boundary torus is a trivial product foliation whose leaves are circles.
\end{defn}

The easiest way to obtain foliated blocks comes from a transversely oriented minimal hyperbolic foliation whose generic leaf is the plane. Take  a closed trasverse curve $\gamma$ to that kind of foliation  and remove the interior of a small tubular neighborhood of $\gamma$, the resulting foliation is a foliated block. 

By minimality, every leaf of a foliated block meets $\gamma$ in a dense set, henceforth the generic leaf of $\FF$ is homeomorphic to the surface $\Sigma$ studied in the previous section. Each leaf of a foliated block intersects the transverse torus in a countable and dense set formed by countably many circles which are the boundary components of the leaf.

Let $\FF^\triangleright$ and $\FF^\triangleleft$ be two foliated blocks as above. To prove the Theorem~\ref{t:thm1} we will realize the prescribed leaf topologies through a $C^\infty$ gluing map between the tori boundaries preserving their trace foliations.

We shall need the following technical lemma that will allow the definition of a suitable gluing map. This lemma is a version of \cite[Proposition 1]{Shiga}, we include the idea of the proof for completeness.

\begin{lemma}\label{l:Cinfty}
Let $X$ and $Y$ be arbitrary dense subsets of $S^1$. Let $x_1,\dots,x_k\in X$ and $y_1,\dots,y_k\in Y$. Assume that there exits a $C^\infty$ map $\varphi:S^1\to S^1$, such that $\varphi(x_i)=y_i$ for $i=1,\dots,k$. Let $x_{k+1}\in X$ be another point different from $x_i$ for $1\leq i\leq k$. For all $\varepsilon>0$ there exists $y_{k+1}\in Y$ and a $C^\infty$ diffeomorphism $\varphi_\varepsilon:S^1\to S^1$ such that $\varphi_\varepsilon(x_i)=y_i$ for $i=1,\dots,k+1$ and $\|\varphi-\varphi_\varepsilon\|_{k+1} <\varepsilon$, where $\|\cdot\|_k$ denotes the usual $k$-norm.
\end{lemma}
\begin{proof}
There exits a $C^\infty$ diffeomorphism $s_\varepsilon:S^1\to S^1$ that it is $\delta$ close to the identity in the $C^{k+1}$ topology,  $s_\varepsilon(y_i)=y_i$ for $i=1,\dots,k$, $s_\varepsilon(\varphi(x_{k+1})) \in Y$ and $\|\varphi - s_\varepsilon\circ\varphi\| <\varepsilon$, define $y_{k+1} =s_\varepsilon(\varphi(x_{k+1}))$.

It is clear that such a function $s_\varepsilon$ does exist, just define a $C^\infty$ bump function $b:S^1\to\R$ that is sufficiently close to cero in the $C^{k+1}$ topology and it is supported in a small neighborhood of $\varphi(x_{k+1})$ that does not contain $\varphi(x_1),\dots,\varphi(x_k)$. Let us consider the flow $\phi_t:S^1\to S^1$ associated to the vector field $b\cdot\partial\theta$. By the density of $Y$, there exists a sequence $t_n\to 0$, with $t_n>0$, such that $\phi_{t_n}(\varphi(x_{k+1}))\in Y$. Just take $s_\varepsilon = \phi_{t_n}$ for some $n$ sufficiently large such that $\|\varphi-\phi_{t_n}\circ\varphi\|_{k+1}<\varepsilon$.

Under these choices $\varphi_\varepsilon = s_\varepsilon\circ \varphi$.
\end{proof}

\begin{rem}
The previous Lemma also holds for any manifold (with obvious modifications for the noncompact case), details can be found in \cite{Shiga}.
\end{rem}

\section{Realizing the topologies}

A \textbf{generic} leaf of $\FF^\triangleright$ and $\FF^\triangleleft$ will be denoted by $L^\triangleright$ and $L^\triangleleft$ respectively, all of them are homeomorphic to $\Sigma$. It will be always assumed that the orientations of the foliated blocks are chosen in order to get opposite orientations in the transverse boundaries (thus the leaves obtained after attaching both foliated blocks are oriented). The transverse regularity of the foliated blocks will be assumed to be $C^\infty$.

Let us identify the boundary tori of $\FF^\triangleright$ and $\FF^\triangleleft$ with $S^1\times T^\triangleright$ and $S^1\times T^\triangleleft$ respectively, where $T^\triangleright$ and $T^\triangleleft$ are smooth transverse circles and the trace foliations are identified with the product foliation on these tori. It is clear that any leaf preserving homeomorphism $\tilde{f}: S^1\times T^\triangleright\to S^1\times T^\triangleleft$ induces a homeomorphism $f: T^\triangleright\to T^\triangleleft$ and reciprocally. We shall assume usual circle parametrizations on both transversals $T^\ast$, $\ast\in\{\triangleright,\triangleleft\}$. 

\begin{proof}[\bf{Proof of Theorem}~\ref{t:thm1}]
Let $S$ be a  noncompact surface satisfying Condition ($\star$). By means of Proposition~\ref{p:surface construction}, there exists two countable collections $\cL^\triangleright=\{L^\triangleright_{n}\mid\ n\in I^\triangleright\}$ and $\cL^\triangleleft=\{L^\triangleleft_{n}\mid\ n\in I^\triangleleft\}$ of manifolds homeomorphic to $\Sigma$ (the generic leaf) and a bijection $\delta:I^\triangleright\times\N\to I^\triangleleft\times\N$ such that the manifold $S_\delta$ is homeomorphic to $S$ for any arbitrary enumeration $\{B^\ast_{i,j}\}_{n\in\N}$ of boundary components of each $L^\ast(i)$ respectively). Identify $\cL^\triangleright$ and $\cL^\triangleleft$ with countable collections of generic leaves in $\FF^\triangleright$ and $\FF^\triangleleft$ respectively. Set $k(i,j)=\pi_1(\delta(i,j))$ and $\ell(i,j)=\pi_1(\delta^{-1}(i,j))$, where $\pi_1$ is the projection onto the first factor. 

Set $X_i=L^\triangleright_i\cap T^\triangleright$, for $i \in I^\triangleright$, and $Y_i=L^\triangleleft_i\cap T^\triangleleft$ $i \in I^\triangleleft$ , and set $X=\bigsqcup_i X_i$ and $Y=\bigsqcup_i Y_i$, all these sets are dense in the respective transverse circles by the minimality of the foliated blocks. Both $X_i$ and $Y_i$ are countable so we can choose good orders on each one of these sets, i.e., an initial enumeration that we do not explicit with subindices as we will only make use of the induced good order.

Let us define a $C^\infty$ diffeomorphism $f:T^\triangleright\to T^\triangleleft$ and enumerations $X_i =\{x_{i,j}\mid\ j\in\N\}$, for $i\in I^\triangleright$, and $Y_i=\{y_{i,j}\mid\ j\in\N\}$, for $i\in I^\triangleleft$, such that $f(x_{i,j})=y_{\delta(i,j)}$.

The construction will be inductive.  Let $x_{1,1}\in X_1$ and $y_{\delta(1,1)}\in Y_{k(1,1)}$ be the first elements of $X_{1}$ and $Y_{k(1,1)}$ respectively (relative to the previously chosen good order) and let us take an arbitrary $C^\infty$ diffeomorphism (for instance a rotation) such that $f_1(x_{1,1})=y_{\delta(1,1)}$.

If $\delta(1,1) = (1,1)$ then set $f_2=f_1$, otherwise set $y_{1,1}\in Y_1$ as the first element of $Y_1$ different from $y_{\delta(1,1)}$. By the density of the sets $X_i$ and Lemma~\ref{l:Cinfty}, there exists $x\in X_{\ell(1,1)}$ and a $C^\infty$ diffeomorphism $f_{2,x}$ such that $f_{2,x}^{-1}(y_{1,1})=x$, $f_{2,x}(x_{1,1})=y_{\delta(1,1)}$ and $\|f_1 - f_{2,x}\|_2<1/2$. Define $x_{\delta^{-1}(1,1)}$ as the first element in $X_{\ell(1,1)}$ that satisfies the previous condition and set $f_2 = f_{2,x_{\delta^{-1}(1,1)}}$. Thus $f_2(x_{1,1})=y_{\delta(1,1)}$ and $f_2(x_{\delta^{-1}(1,1)})=y_{1,1}$.

Assume now that $x_{i,j}$, $x_{\delta^{-1}(i,j)}$, $y_{i,j}$ and $y_{\delta(i,j)}$ were defined\footnote{Since the sets $I^\ast$ can be finite, the index $i$ may be undefined when $n>\# I^\ast$, we shall consider implicit that $i\in I^\ast$.} for all $i,j\leq n$ as well as $C^\infty$ circle diffeomorphisms $f_k$ for $1\leq k\leq 2n^2$ such that $f_{2n^2}(x_{i,j})=y_{\delta(i,j)}$ and $f_{2n^2}(x_{\delta^{-1}(i,j)})=y_{i,j}$ for all $i,j\leq n$ and $\|f_{k-1} -f_{k}\|_{k}<\frac{1}{k}$ for all $k\in\{2,\dots,2n^2\}$.

Fix $i=1$, if $(1,n+1)\in\{\delta^{-1}(i,j)\mid\ 1\leq i,j\leq n\}$ then there is nothing to do as both $x_{1,n+1}$ and $y_{\delta(1,n+1)}$ were already defined at this point, set $f_{2n^2+1}=f_{2n^2}$. Let $i_0\leq n+1$ be the minimum integer (if exists) such that $(i_0,n+1)$ does not belong to $\{\delta^{-1}(i,j)\mid\ 1\leq i,j\leq n\}$ and set $f_{2n^2+m}=f_{2n^2}$ for $m\in\{1,\dots,i_0-1\}$. Let $x\in X_{i_0}$ be the minimum element in $X_{i_0}$ which is different to the previously defined ones and declare $x_{i_0,n+1}=x$. By the density of $Y_{k(i_0,n+1)}$ and Lemma~\ref{l:Cinfty}, there exists $y\in Y_{k(i_0,n+1)}$, different from the previously defined ones, and a $C^\infty$ circle diffeomorphism $f_{2n^2+i_0,y}$ such that, $f_{2n^2+i_0}(x_{i,j})=y_{\delta(i,j)}$, $f_{2n^2+i_0}(x_{\delta^{-1}(i,j)})=y_{i,j}$ for all $i,j\leq n$, $f_{2n^2+i_0,y}(x_{i_0,n+1})=y$ and $$\|f_{2n^2}-f_{2n^2+i_0,y}\|_{2n^2+i_0}<\frac{1}{2n^2+i_0}\,.$$ Define $y_{\delta(i_0,n+1)}$ as the minimum element $y\in Y_{k(i_0,n+1)}$ that satisfies these conditions and set $f_{2n^2+i_0}=f_{2n^2+i_0,y_{\delta(i_0,n+1)}}$. Repeat, recursively, this process with any index $(i,n+1)$ where $y_{\delta(i,n+1)}$ was still not defined. The previous inductive reasoning must be applied again to the indices $(n+1,j)$, for $1\leq j\leq n+1$, where $y_{\delta(n+1,j)}$ was still not defined and, of course, whenever $\# I^\triangleright\geq n+1$.

Now, we repeat an analogous argument to define $x_{\delta^{-1}(i,n+1)}$ for $1\leq i\leq n$ and $x_{\delta^{-1}(n+1,j)}$ for $j\in\{1,\dots,n+1\}$ in the case that they were not already defined obtaining a function $f_{2n^2}:S^1\to S^1$ satisfying the desired conditions, details are left to the reader. This shows that the above process is inductive. Observe that the arbitrary good orders chosen at the beginning of the proof on the sets $X_i$ and $Y_i$ guarantee that every point in $X$ (resp. $Y$) has eventually an (unique) image (resp. preimage) in $Y$ (resp. $X$).

The sequence of $C^\infty$ diffeomorphisms $f_k$, $k\in\N$, is Cauchy in the $C^m$ topology for every $m\in\N$ and therefore it converges to a $C^\infty$ diffeomorphism $f:S^1\to S^1$ that satisfies $f(x_{i{\color{red},}j})=y_{\delta(i,j)}$ for all $i,j\in I^\triangleright\times\N$ as desired.

Observe that $x_{i,j}$ (resp. $y_{i,j}$) belongs to a boundary component, that can be denoted by $B^\triangleright_{i,j}$ (resp. $B^\triangleleft_{i,j}$), of $L^\triangleright_i$ (resp. $L^\triangleleft_i$).

As we claimed above, the $C^\infty$ diffeomorphism $f:T^\triangleright\to T^\triangleleft$ also defines a gluing map $\tilde{f}=\id\times f$ between the boundary tori of the foliated blocks (here, leaf boundary components are identified with horizontal fibers of the tori).

The foliation $\FF^{\Join}_f$ obtained from the union of $\FF^\triangleright$ with $\FF^\triangleleft$ using this gluing map contains a leaf which is homeomorphic to $S_\delta$ by construction (see also Remark~\ref{r:independence}).

If we prescribe a countable family of noncompact oriented surfaces $S_n$ satisfying the Condition $(\star)$, then define $S=\bigsqcup_n S_n$ which is a nonconnected noncompact oriented surface that still can be seen as the union of countably many manifolds homeomorphic to $\Sigma$. So the previous construction still applies to get a countable collection of leaves, each one homeomorphic to the connected components of $S$, these are precisely the surfaces $S_n$.

In order to guarantee that the generic leaf is a Cantor tree, we can add a leaf homeomorphic to the Cantor tree to our prescribed leaves. With this addition the generic leaf cannot support handles and hence it is the Cantor tree (plane and cylinder are impossible since the generic leaves of the foliated blocks are homeomorphic to $\Sigma$).
\end{proof}

\begin{proof}[\bf{Proof of Theorem}~\ref{t:thm2}]
Observe first that a leaf Satisfying Condition ($\star\star$) also satisfies Condition ($\star$). If the generic leaf is a Cantor tree  with handles then realize the given surfaces $S_n$ in a foliation $\FF$ given by Theorem~\ref{t:thm1}. Then choose a closed transverse circle to $\FF$, remove a tubular neighborhood of this transversal and attach to the boundary a manifold homeomorphic to $H\times S^1$, where $H$ is a handle (the complement of the interior of a smooth disk in a torus), and trivially foliated by the slices $H\times\{\ast\}$. This surgery attach handles to every end and therefore the generic leaf is transformed into a Cantor tree  with handles. Attaching handles does not change the topology of the leaves homeomorphic to $S_n$ since, by hypothesis, every end of these surfaces is accumulated by genus.

For the case where the generic leaf is homeomorphic to the Loch Ness monster we mimic the previous procedure. By means of \cite[Theorem 1]{Gusmao-Menino} we know that there exists a foliation transversely bi-Lipschitz $\FF$ on a closed (Seifert) $3$-manifold whose generic leaf is the plane and contain a countable collection of leaves $L_n$ each one homeomorphic to $S_n$ respectively. Performing the same surgery as before we obtain the desired foliation whose generic leaf is the Loch Ness Monster.
\end{proof}

\section{Final Comments}

Relative to the ambient topology of the constructed foliations, observe that it will obviously depend on the topology of the foliated blocks. For instance, if the foliated blocks are obtained  by removing a tubular neighborhood of a regular fiber on a foliation transverse to the fibers of a Seifert manifold then the resuting manifold will be also Seifert (with non maximal Euler number). Usually our construction provides graph manifolds, for instance, we can perform {\em Hirsch surgeries} as described in \cite[Subsection 3.3]{ADMV} with gluing maps as the constructed in this work to get minimal hyperbolic foliations in a graph manifold with countably many leaves homeomorphic to any prescribed family of surfaces satisfying the Condition $(\star)$ (generic leaf in this case is the Cantor tree).

Observe also that our construction also works when the trace foliation of the boundary tori is not trivial but still given by circles (a linear foliation associated to a rational rotation). In this case the gluing map is a perturbation of a suitable Dehn map between the boundary tori.

The construction can be also applied to foliated blocks with more boundary components, in this case we can make the construction with gluing maps between these boundary tori or construct minimal hyperbolic foliations on graph manifolds with more pieces. Morover we can apply the procedure to foliated blocks whose generic leaf is the Loch Ness monster punctured along countably many balls (instead of the surface $\Sigma$), with this kind of foliated blocks we can only construct foliations whose generic leaves satisfy the Condition $(\star\star)$. This, in principle, extends the family of ambient $3$-manifolds where Theorem~\ref{t:thm2} applies.

Our construction does not work when the trace foliation of the boundary tori have noncompact orbits. This would imply that the generic leaf of the foliated block would have noncompact boundary components. Although we can still obtain many interesting topologies by gluing these kind of generic leaves (as it is done in \cite{Gusmao-Menino}), the main point is that Remark~\ref{r:independence} does not apply in this context as arbitrary reorderings of noncompact boundary components cannot be realized in general by a homeomorphism of the surface. This is one of the main differences between the arguments in \cite{Gusmao-Menino} and those given in this work.

Relative to the transverse regularity, recall that the construction given in \cite{Gusmao-Menino} cannot be realized in regularity $C^2$, since one of the cases of Theorem~\ref{t:thm2} depends on this result, it inherits this restriction in its transverse regularity.  For Theorem~\ref{t:thm1}, it seems hard to improve the regularity to real analytic and this would depend in a real analytic version of Lemma~\ref{l:Cinfty}. The existence of a critical regularity as well as topological contraints for the coexitence of some topologies as leaves of the same codimension one minimal hyperbolic foliation remains as an open problem for these cases.

\begin{ack}
The second author wants to thank the Programa Nosso Cientista Estado do Rio de Janeiro 2015-2018 FAPERJ-Brazil (``Din\^amicas n\~ao hiperb\'olicas''), MathAmSud 2019-2020 CAPES-Brazil (``Rigidity and Geometric Structures on Dynamics''), the CNPq research grant 310915/2019-8 and the Ministerio de Ciencia e Innovaci\'on (Spain) grant PID2020-114474GB-I00 that partially supported this research. 
\end{ack}

\end{document}